\documentclass[12pt]{article}
\usepackage[usenames,dvipsnames]{xcolor}

\usepackage[affil-it]{authblk}
\usepackage{amsmath,amssymb,amsfonts,mathrsfs,amsthm,pgf,tikz}
\usetikzlibrary{shapes,positioning,decorations}
\usetikzlibrary{decorations.pathreplacing}
\usepackage[bookmarks=true,pdfborder={0 0 0}]{hyperref}
\usepackage{soul}

\newtheorem{theorem}{Theorem}[section]
\newtheorem{lemma}[theorem]{Lemma}
\newtheorem{corollary}[theorem]{Corollary}

\newtheorem{obs}[theorem]{Observation}
\newtheorem{prop}[theorem]{Proposition}
\newtheorem{question}[theorem]{Question}

\theoremstyle{definition}

\newtheorem{example}[theorem]{Example}

\theoremstyle{remark}

\numberwithin{equation}{section}
\DeclareMathOperator{\rank}{rank}
\DeclareMathOperator{\nullity}{nullity}

\DeclareMathOperator{\n}{N}

\DeclareMathOperator{\mr}{mr}

\newcommand{\floor}[1]{\Big\lfloor #1 \Big\rfloor}
\newcommand{\ST}{\;\big|\;}

\author[1]{Keivan Hassani Monfared}
\affil[1]{\small Department of Mathematics and Statistics, University of Calgary, 2500 University Drive NW, Calgary, AB, T2N 1N4, Canada\\
k1monfared@gmail.com}

\author[2]{Sudipta Mallik}
\affil[2]{ Department of Mathematics and Statistics, Northern Arizona University, 801 S. Osborne Dr. PO Box: 5717, Flagstaff, AZ 86011, USA\\
sudipta.mallik@nau.edu}

\begin{document}
\title{The maximum multiplicity of an eigenvalue of symmetric matrices with a given graph}

\maketitle

\begin{abstract}For a graph $G$, $M(G)$ denotes the maximum multiplicity occurring of an eigenvalue of a symmetric matrix whose zero-nonzero pattern is given by the edges of $G$. We introduce two combinatorial graph parameters $T^-(G)$ and $T^+(G)$ that give  a lower and an upper bound for  $M(G)$ respectively, and we show that these bounds are sharp.
\end{abstract}

\renewcommand{\thefootnote}{\fnsymbol{footnote}} 
\footnotetext{\emph{2010 Mathematics Subject Classification. 05C50,65F18\\ Keywords: Symmetric Matrix, Graph, Tree, Eigenvalue, Path Cover.}}    
\renewcommand{\thefootnote}{\arabic{footnote}} 

\section{Introduction}
For an $n\times n$ symmetric matrix $A=[a_{ij}]$, the {\it graph of $A$}, denoted by $G(A)$, is the simple graph on $n$ vertices $1,2,\ldots,n$ where $\{i,j\}$ is an edge of $G(A)$ if and only if $a_{ij}\neq 0$ for $i\neq j$. For a graph $G$ on $n$ vertices, $S(G)$ denotes the set of all $n\times n$ real symmetric matrices whose graph is $G$, and $M(G)$ denotes the maximum multiplicity occurring of an eigenvalue of a matrix in $S(G)$. The {\it minimum rank} of $G$, denoted by $\mr(G)$, is the minimum rank of $A$ where $A$ runs over $S(G)$. Note that if the multiplicity of an eigenvalue $\lambda$ is $k$ for some matrix $A$ in $S(G)$, then the nullity of $A-\lambda I$ is $k$ which implies 
$\mr(G)\leq \rank(A-\lambda I)= n-k.$
So we can conclude that
\[ M(G)=\displaystyle\max_{A\in S(G)}\nullity(A)=n-\mr(G). \]

There is a lot of interest in determining the maximum multiplicity of eigenvalues of matrices whose graph is given \cite{vibrationtrees74, johnsonduarte99, johnsonduarte02, johnsonsaiago02, nylenminrank96, parter60, wiener84}. The {\it path cover number} of a graph $G$, denoted by $P(G)$, is the minimum number of vertex-disjoint paths needed as induced subgraphs of $G$ that cover all the vertices of $G$. Duarte and Johnson in their 1999 paper \cite{johnsonduarte99} introduced a graph parameter $\Delta(T)$ for a tree $T$ to be 
\begin{align*}
\Delta(T) := \max \{ p - q \ST &\text{there exist } q \text{ vertices of } T \text{ whose } \\& \text{deletion leaves } p \text{ vertex-disjoint paths} \},
\end{align*} 
and showed that $\Delta(T)$ is equal to $M(T)$ and $P(T)$:

\begin{theorem}\label{tree}\cite{johnsonduarte99}
For all trees $T$, $M(T)=P(T)=\Delta(T)$.
\end{theorem}

The definition  of $\Delta$ can be extended to any graph $G$. The proof of Duarte and Johnson shows that for any graph $G$, $\Delta(G)$ is a lower bound for $P(G)$ and $M(G)$:

\begin{theorem}\cite{bfh04,johnsonduarte99}\label{knownbounds}
For all graphs $G$, $\Delta(G)\leq P(G)$ and $\Delta(G)\leq M(G)$.
\end{theorem}

Later in 2004 Barioli, Fallat, and Hogben \cite{bfh04} pushed the results further and provided an algorithm to compute $\Delta$. Note from Theorem \ref{knownbounds} that $M(G)$ and $P(G)$ are both upper bounds for $\Delta(G)$. But they have no relationship in general. This was observed by Barioli, Fallat, and Hogben \cite[Figures 1 and 2]{bfh05} in the following examples: For the wheel graph $W_5$ we have \[P(W_5)=2<3=M(W_5)\] and for the $5$-sun $H_5$ we have \[M(H_5)=2<3=P(H_5).\]

From the definition of $M(G),P(G)$, and $\Delta(G)$, it follows that they can be computed componentwise for a disconnected graph.
\begin{obs}\label{components}
Let $G$ be a graph with $k$ connected components $G_1$, $G_2$, $\ldots$, $G_k$. Then 
$M(G)=\displaystyle\sum_{i=1}^k M(G_i)$, 
$P(G)=\displaystyle\sum_{i=1}^k P(G_i)$, and 
$\Delta(G)=\displaystyle\sum_{i=1}^k \Delta(G_i)$.
\end{obs}

Using the preceding observation, Theorem \ref{tree} can be extended to forests.

\begin{theorem}\label{forest}
If $G$ is a forest, then $\Delta(G) = P(G) = M(G)$.
\end{theorem}

Note that the converse of Theorem \ref{forest} is not true. 
\begin{example}\label{unicyclic} Consider the unicyclic graph $G$ in Figure \ref{unicyclicgraph}. We can verify that $M(G)=P(G)=\Delta(G)=2$. 

The preceding example shows that the equalities in Theorem \ref{knownbounds} occur for some graphs with cycles in addition to trees. Indeed, for any graph $G$ in the the following infinite family of unicyclic graphs (see Figure \ref{unicyclicclass}) we have $M(G)=P(G)=\Delta(G)=2$. 

Let $P$ be a path on at least 5 vertices. Pick any three non-pendant consecutive vertices on $P$, say $u, v$ and $w$. Now $G$ is obtained from $P$ by appending a path of length at least 2 from $u$ to $w$. Clearly $P(G)=2$. Deleting $u$ and $w$ we have $\Delta(G) =4-2= 2$. Furthermore, note that each $G$ on $n$ vertices in this family  has an induced $P_{n-1}$. Hence $\mr(G) \geq \mr(P_{n-1}) = n-2$. But $\mr(G) < n -1$, since $G$ is not a path \cite[Cor 1.5]{minrank07}. This shows $\mr(G) = n-2$, thus $M(G) = 2$.
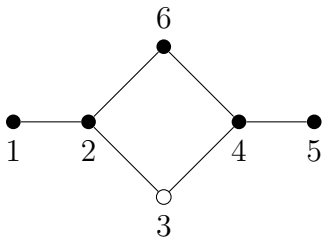
\begin{figure}
\begin{center}
\begin{tikzpicture}[colorstyle/.style={circle, fill, black, scale = .5}]
	
	\node (6) at (0,1)[colorstyle, label=above:$6$]{};
	\node (3) at (0,-1)[circle, fill, white, scale = .5]{};	
	\node (3) at (0,-1)[circle, draw, black, scale = .5, label=below:$3$]{};
	\node (2) at (-1,0)[colorstyle, label=below:$2$]{};
	\node (1) at (-2,0)[colorstyle, label=below:$1$]{};
	\node (4) at (1,0)[colorstyle, label=below:$4$]{};
	\node (5) at (2,0)[colorstyle, label=below:$5$]{};

	\draw [] (1)--(2)--(3)--(4)--(5);
	\draw [] (4)--(6)--(2);
\end{tikzpicture}
\caption{Graph $G$ with $M(G)=P(G)=\Delta(G)=2$}\label{unicyclicgraph}
\end{center}
\end{figure}
\end{example}

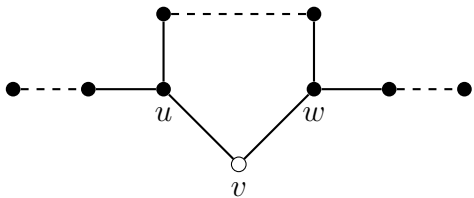
\begin{figure}
\begin{center}
\begin{tikzpicture}[colorstyle/.style={circle, fill, black, scale = .5}]
	
	\node (6) at (-1,1)[colorstyle]{};
	\node (8) at (1,1)[colorstyle]{};
	\node (3) at (0,-1)[circle, fill, white, scale = .5]{};	
	\node (3) at (0,-1)[circle, draw, black, scale = .5, label=below:$v$]{};
	\node (2) at (-1,0)[colorstyle, label=below:$u$]{};
	\node (1) at (-2,0)[colorstyle]{};
	\node (0) at (-3,0)[colorstyle]{};
	\node (4) at (1,0)[colorstyle, label=below:$w$]{};
	\node (5) at (2,0)[colorstyle]{};
	\node (7) at (3,0)[colorstyle]{};

	\draw[thick] (1)--(2)--(3)--(4)--(5);
	\draw[thick] (2)--(6);
	\draw[thick] (4)--(8);
	\draw[thick,dashed] (6)--(8);
	\draw[thick,dashed] (0)--(1);
	\draw[thick,dashed] (5)--(7);
\end{tikzpicture}
\caption{Graph $G$ with $M(G)=P(G)=\Delta(G)=2$}\label{unicyclicclass}
\end{center}
\end{figure}

	
%
	
%

In 2007 Fernandes \cite{fernandes07} expressed $M(G)$ for some unicyclic graphs $G$ in terms of certain graph parameters. In 2008 AIM Minimum Rank Work Group \cite{aim08} introduced the zero forcing number $Z(G)$ for a graph $G$ and proved that $M(G)\leq Z(G)$ for all graphs $G$, where the equality holds for forests.
In this article we introduce new combinatorial bounds for $M(G)$. Motivated by the definition of $\Delta(G)$, in Section \ref{sectiondeltaplus} we introduce a graph parameter $\Delta^+(G)$ in terms of path covers of $G$ and show that \[\Delta(G) \leq M(G)\leq \Delta^+(G),\] for all graphs $G$. Then in Section \ref{sectionteeplusandteeminus} we introduce two more parameters $T^-(G)$ and $T^+(G)$ in terms of tree covers of $G$, and show that \[T^-(G) \leq M(G)\leq T^+(G),\] for all graphs $G$ and that the bounds are sharp. In Section \ref{sectioncomputing} we reduce the computation time for $T^-$ and $T^+$ by finding an optimal set of vertices of small size.
Finally we pose some open problems in Section \ref{sectionproblems}.

\section{Graph Invariant $\Delta^+(G)$}\label{sectiondeltaplus}
For a graph $G$, we define $\Delta^+(G)$ to be the minimum of $p+q$ when deletion of $q$ vertices from $G$ leaves $p$ vertex-disjoint paths. 
\begin{obs}
For any graph $G$ on $n$ vertices, $\Delta^+(G) \leq n$.
\end{obs}
\begin{proof}
Let $S$ be an optimal set of $q$ vertices for $\Delta^+(G)$. That is, deleting the $q$ vertices in $S$ from $G$ leaves $p$ disjoint paths such that $p+q$ is minimum. Since each path has at least one vertex, $p\leq n-q$ and  then $\Delta^+(G) = p + q \leq n$. 
\end{proof}

The following examples compute $\Delta^+$ for some families of graphs.

\begin{example}\label{starex}
For the star $S_n$ on $n \geq 4$ vertices, $M(S_n) = \Delta^+(S_n) = n-2$. 

Note that $\mr(S_n) = 2$ \cite[Obs 1.2]{minrank07} which implies $M(S_n) = n-2$. Also deleting $n-3$ pendant vertices from a star leaves a path, viz., $P_3$. Hence $\Delta^+(S_n)=1+(n-3) = n-2$.
\end{example}

\begin{example}\label{cycle}
For the cycle $C_n$ on $n$ vertices, $M(C_n) = \Delta^+(C_n) = 2$. 

Note that $\mr(C_n) = n-2$ \cite[Obs 1.6]{minrank07}, hence $M(C_n) = 2$. Also note that to get paths induced in $C_n$, we need to delete at least one vertex. If deletion of $q\geq 1$ vertices from $C_n$ gives $p$ paths $P_{n_1},\ldots,P_{n_p}$, then $1\leq p\leq q$. Thus $p+q\geq 2$, where equality holds if and only if the number of optimal vertices deleted is $1$. Thus $\Delta^+(C_n) = 1+1=2$.
\end{example}

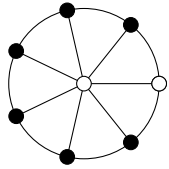
\begin{figure}[h]
\begin{center}
\begin{tikzpicture}
	\def \n {7}
	\def \radius {1cm}

  	\draw[-] (0:\radius) arc (0:360:\radius); 

	\foreach \s in {1,...,\n} 
	{
  	\draw[-] (0,0) -- ({360/\n * (\s )}:{\radius});
  	\node[draw, circle, fill, black, scale = .5] at ({360/\n * (\s - 1)}:\radius) {};
 	}
 			
	\node[circle, fill, white, scale = .5] at (0,0) {}; 
	\node[draw, black, circle, scale = .5] at (0,0) {}; 
	
	\node[circle, fill, white, scale = .5] at (0:\radius) {};
	\node[draw, black, circle, scale = .5] at (0:\radius) {};
	
\end{tikzpicture}
\end{center}
\caption{Wheel graph $W_8$}\label{figwheel}
\end{figure}
\begin{example}\label{wheelex}
For the wheel $W_n$ on $n\geq 4$ vertices, $M(W_n) = \Delta^+(W_n) = 3$. 

Note that $\mr(W_n)\geq \mr(C_{n-1})=n-3$, but it cannot be more than $n-3$ since it is neither a path nor a 2-connected linear 2-tree \cite[Cor 1.5, Thm 2.26]{minrank07}. Hence $\mr(W_{n}) = n-3$ and consequently $M(W_n) = 3$. For $\Delta^+(W_n)$, delete the  vertex of degree $n-1$ and another vertex of degree 3 (see white vertices in Figure \ref{figwheel}) to get $P_{n-2}$. So we have $\Delta^+(W_n)=1+2=3$.
\end{example}

\begin{figure}[h]
\begin{center}
\begin{tikzpicture}
	\def \n {9}
	\def \radius {.7cm}

  	\draw[-] (0:\radius) arc (0:360:\radius); 

	\foreach \s in {1,...,\n} 
	{
  	\draw[-] ({360/\n * (\s )}:{2*\radius}) -- ({360/\n * (\s )}:{\radius});
  	\node[draw, circle, fill, black, scale = .5] at ({360/\n * (\s - 1)}:\radius) {};
  	\node[circle, fill, white, scale = .5] at ({360/\n * (\s - 1)}:{2*\radius}) {};
  	\node[draw, black, circle, scale = .5] at ({360/\n * (\s - 1)}:{2*\radius}) {};
 	}
	
	\node[circle, fill, white, scale = .5] at (0:\radius) {};
	\node[draw, black, circle, scale = .5] at (0:\radius) {};
	
	\node[draw, circle, fill, black, scale = .5] at (0:{2*\radius}) {};
	\node[draw, circle, fill, black, scale = .5] at ({360/\n * (2 - 1)}:{2*\radius}) {};
	\node[draw, circle, fill, black, scale = .5] at ({360/\n * (\n - 1)}:{2*\radius}) {};

\end{tikzpicture}
\end{center}
\caption{The 9-sun $H_9$}\label{figsun}
\end{figure}
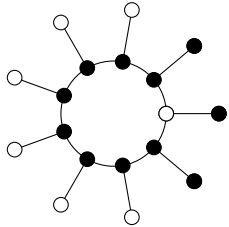
\begin{example}
Let $H_n$ be the $n$-sun. Then
\begin{enumerate}
\item[(a)] $M(H_n)=\left\{\begin{array}{cl}
2&\text{ if } n=3\\
\displaystyle\floor{\frac{n}{2}}&\text{ if } n\geq 4
\end{array}\right.$ 
\item[(b)] $\Delta^+(H_n) \leq n \text{ for } n\geq 3. $
\end{enumerate}

Part (a) is shown in \cite[Prop 3.1]{bfh05}. For part (b), consider $H_3$ first. Note that deleting a vertex from the cycle leaves two paths, viz., $P_1$ and $P_4$. Thus $\Delta^+(H_3) \leq 2 + 1 =3$. Now consider $H_n$ where $n\geq 4$. Note that deleting a vertex from the cycle and all the pendant vertices that are at distance more than 2 from it (i.e., a total of $n-2$ vertices)  leaves two paths, viz., $P_1$ and $P_{n+1}$. Thus $\Delta^+(H_n) \leq 2 + (n-2) = n$. For an example see the 9-sun $H_9$ in Figure \ref{figsun}.
\end{example}

Next we show that $\Delta^+(G)$ is an upper bound for $M(G)$ for any graph $G$. First we need the following lemma derived from the definition of $M(G)$.
\begin{lemma}\label{mult}
Let $G$ be a graph on $n$ vertices. Then for all $A\in S(G)$, $$\rank(A)\geq \mr(G)=n-M(G).$$
\end{lemma}

\begin{theorem}\label{delta}
For all graphs $G$, \[M(G)\leq \Delta^+(G).\] 
\end{theorem}
\begin{proof} 
Let $G$ be a graph on $n$ vertices. To show $M(G)\leq \Delta^+(G)$, we show that $\mr(G)=n-M(G)\geq n-\Delta^+(G)$. Let $A\in S(G)$. It suffices to show that $\rank(A)\geq n-\Delta^+(G)$.  Let $P_{n_1},\ldots,P_{n_p}$ be the vertex-disjoint paths remaining after deletion of an optimal $q$ vertices from $G$ such that $\Delta^+(G)=p+q$ where $n-q=n_1+\cdots+n_p$. For $i=1,\ldots,p$, let $B_i$ be the principle submatrix of $A$ such that the graph of $B_i$ is $P_{n_i}$. So $B_1\oplus\cdots\oplus B_p$ is an $(n-q)\times (n-q)$ principle submatrix of $A$ and  
$$\rank(A)\geq \rank(B_1\oplus\cdots\oplus B_p)=\sum_{i=1}^p\rank(B_i).$$

By Theorem \ref{tree}, $M(P_{n_i})=1$ for $i=1,\ldots,p$.
By Lemma \ref{mult}, $\rank(B_i)\geq |P_{n_i}|-M(P_{n_i})=n_i-1$ for $i=1,\ldots,p$. Thus 
\begin{align*}
\rank(A)\geq \sum_{i=1}^p\rank(B_i)&\geq \sum_{i=1}^p(n_i-1)\\
&=\sum_{i=1}^p n_i-\sum_{i=1}^p 1\\
&=(n-q)-p\\
&=n-(p+q)\\
&=n-\Delta^+(G).
\end{align*}
\vspace*{-12pt}
\end{proof}

From Theorem \ref{knownbounds} and \ref{delta}, we achieve the following upper and lower bounds for $M(G)$:
\begin{corollary}\label{Mbounds}
$\Delta(G)\leq M(G)\leq \Delta^+(G)$ for all graphs $G$.
\end{corollary}


While Theorem \ref{forest} asserts that for any forest $G$, $\Delta(G) = M(G)$, it is easy to see that even for trees $\Delta^+$ and $M$ do not necessarily coincide. 

\begin{figure}[h]
\begin{center}
\begin{tikzpicture}
	\def \n {3}
	\def \radius {.7cm}

	\foreach \s in {1,...,\n} %
	{
  	\draw[-] (0:0) -- ({360/\n * (\s )}:{2*\radius});
  	\node[circle, fill, black, scale = .5] at ({360/\n * (\s - 1)}:\radius) {};
  	\node[black, fill, circle, scale = .5] at ({360/\n * (\s - 1)}:{2*\radius}) {};
 	}
	
	\node[circle, fill, black, scale = .5] at (0:0) {};
	\node[circle, fill, white, scale = .5] at ({360/\n * (1 - 1)}:\radius) {};
	\node[circle, draw, black, scale = .5] at ({360/\n * (1 - 1)}:\radius) {};

\end{tikzpicture}
\end{center}
\caption{A generalized star}\label{figgenstar}
\end{figure}
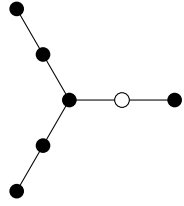
\begin{example}\label{genstar}
Consider the generalized star $G$ shown in Figure \ref{figgenstar}. Deleting the white vertex we have $\Delta^+(G) = 3$, but $M(G) = P(G) = 2$.
\end{example}

Examples \ref{unicyclic} and \ref{genstar} suggest that the answer to the following questions are not trivial.

\begin{question} For what graphs $G$, $\Delta(G) = M(G)$?\end{question}
The answer to this question shall contain all forests (by Theorem \ref{forest}) and the unicyclic graphs introduced in Example \ref{unicyclic}.
\begin{question} For what graphs $G$, $M(G) = \Delta^+(G)$?\end{question}
The answer does not include all forests as shown in Example \ref{genstar}, but includes stars, cycles, and wheels (Example \ref{starex}, \ref{cycle}, \ref{wheelex}).
\begin{question}For what graphs $G$, $\Delta(G) = M(G) = \Delta^+(G)$?\end{question}
The answer shall include disjoint unions of stars (see Example \ref{starex}) and paths.

\section{Graph Invariants $T^-(G)$ and $T^+(G)$}\label{sectionteeplusandteeminus}

Recall that the definitions of $\Delta$ and $\Delta^+$ for a graph $G$ involve induced paths obtained by deleting vertices from $G$. One of the reasons for considering induced paths is that an eigenvalue of a symmetric matrix whose graph is a path has maximum multiplicity one \cite{johnsonduarte99}. We investigate if replacement of paths by other graphs in the definitions of $\Delta$ and $\Delta^+$ improves the bounds for $M(G)$. So we define two new graph parameters $T^-$ and $T^+$ as follows:    

\begin{align*}
T^-(G) := \max \Biggl\{ P(G\setminus S)-|S| \ST   & S  \text{ is a subset of vertices of } G \\ 
 & \text{ such that }  G\setminus S \text{ is a forest } \Biggr\},\\
T^+(G) := \min \Biggl\{ P(G\setminus S)+|S| \ST   & S  \text{ is a subset of vertices of } G \\ 
 & \text{ such that }  G\setminus S \text{ is a forest } \Biggr\}.
\end{align*}
Assume $S$ is a set of $q$ vertices, and $G\setminus S$ is a forest which is a vertex-disjoint union of $p$ trees $T_1, T_2,\ldots, T_p$. Then $P(G\setminus S) = \sum_{i=1}^p P(T_i)$ by Observation \ref{components}. Hence $T^-(G)$ and $T^+(G)$ can be rewritten as the following:
\begin{small}\begin{align*}
T^-(G) = \max \Biggl\{  \left( \sum_{i=1}^{p} P(T_i) \right) - q \ST   & \text{there exist } q \text{ vertices of } G \text{ whose deletion } \\ 
 &  \text{leaves } p \text{ vertex-disjoint trees } T_1,\ldots,T_p \Biggr\},\\
T^+(G) = \min \Biggl\{ \left( \sum_{i=1}^{p} P(T_i) \right) + q \ST  & \text{there exist } q \text{ vertices of } G \text{ whose deletion } \\ & \text{leaves } p \text{ vertex-disjoint trees } T_1,\ldots,T_p \Biggr\}.
\end{align*}\end{small}
For a forest $G$, the optimal set of vertices to be deleted is the empty set (i.e., $q=0$)
 and consequently  $T^+(G)=T^-(G)=P(G)$.
 
  The following examples compute $T^-$ and $T^+$ for some other  families of graphs.

\begin{example}\label{tpluscycle}
$T^-(C_n)=0$ and $T^+(C_n)=2$.

Note that to get trees induced in $C_n$, we need to delete at least one vertex. If deletion of $q\geq 1$ vertices from $C_n$ gives $p$ trees (paths) $P_{n_1},\ldots,P_{n_p}$, then $1\leq p\leq q$. Thus $q+\sum_{i=1}^p P(P_{n_i})=q+p\geq 2$ and  $-q+\sum_{i=1}^p P(P_{n_i})=-q+p\leq 0$, where equalities hold if and only if the number of optimal vertices deleted is $1$. Thus $T^+(C_n)=2$ and $T^-(C_n)=0$.
\end{example}

\begin{example}\label{wheel}
Let $W_n$ be the wheel graph on $n$ vertices. Then
\begin{enumerate}
\item[(a)] $M(W_n)=3.$ 
\item[(b)] $T^-(W_n)=-1,T^+(W_n)=3.$
\end{enumerate}

The set of optimal vertices to be deleted for $T^-$, $T^+$, and $\Delta^+$ are shown as white vertices in Figure \ref{figwheel}. 
\end{example}

\begin{example}
Let $H_n$ be the $n$-sun. Then
\begin{enumerate}
\item[(a)] $M(H_n)=\left\{\begin{array}{cl}
2&\text{ if } n=3\\
\displaystyle\floor{\frac{n}{2}}&\text{ if } n\geq 4
\end{array}\right.$ 
\item[(b)] $T^-(H_n)=T^+(H_n)-2=\left\{\begin{array}{cl}
1&\text{ if } n=3\\
\displaystyle\floor{\frac{n}{2}}&\text{ if } n\geq 4
\end{array}\right.$
\end{enumerate}

An optimal set of vertices to be deleted for $T^-$ and $T^+$ is any single vertex from the cycle. 
\end{example}

\begin{obs}
The optimal sets for $T^-(G)$ and $T^+(G)$ can be chosen to be the same, when $G$ is a wheel graph or the $n$-sun.
\end{obs}

By definitions we have the following results.
\begin{prop}\label{prop1} Let $G$ be a graph. Then
\begin{enumerate}
\item[(a)] \label{propparta} $\Delta(G)\leq T^-(G)$ and
\item[(b)] \label{proppartb} $P(G)\leq T^+(G)\leq \Delta^+(G)$.
\end{enumerate}
\end{prop}
\begin{proof}
(a) Suppose that deletion of $q$ vertices from $G$ leaves $p$ vertex-disjoint paths $P_{n_1},\ldots,P_{n_p}$ such that $\Delta(G)=p-q$. Since $P(P_{n_i})=1$, $$\Delta(G)=-q+\sum_{i=1}^{p}P(P_{n_i})\leq T^-(G).$$
(b) Suppose that deletion of $q$ vertices $v_1,\ldots,v_q$ from $G$ leaves $p$ vertex-disjoint trees $T_1,\ldots,T_p$ such that $T^+(G)=q+\sum_{i=1}^{p}P(T_i)$. Note that union of $v_1,\ldots,v_q$, and optimal path covers of $T_1,\ldots,T_p$ forms a path cover of $G$ of length $T^+(G)=q+\sum_{i=1}^{p}P(T_i)$. Thus $P(G)\leq T^+(G)$.
\\

Let $Q$ be an optimal set of $q$ vertices such that deleting them from $G$ leaves $p$ disjoint paths $P_1, P_2, \ldots, P_p$; and $\Delta^+(G) = p+q$. Since $P(P_i) = 1$,
\begin{align*}
\Delta^+(G) &= p+q\\
&= q + \sum_{i=1}^{p} 1\\
&= q + \sum_{i=1}^{p} P(P_i)\\
&\geq T^+(G).
\end{align*}
\vspace*{-12pt}
\end{proof}

First we note that the equality in Proposition \ref{prop1}(a) holds for forests. In fact we can show the equality for all graphs.

\begin{theorem}\label{prop2}
$\Delta(G)=T^-(G)$ for all graphs $G$.
\end{theorem}
\begin{proof}
Let $G$ be a graph. By Proposition \ref{prop1} we have $\Delta(G)\leq T^-(G)$. So it suffices to show that $\Delta(G)\geq T^-(G)$. Note that $\Delta(T) = P(T)$, for any tree $T$.

Now choose an optimal set of $q$ vertices for $T^-(G)$ such that deleting them leaves $p$ vertex-disjoint trees $T_1, T_2, \ldots, T_p$.  For each $i=1,2,\ldots, p$, choose an optimal set of $k_i$ vertices for $\Delta(T_i)$ such that deleting them from $T_i$ leaves $\ell_i$ vertex-disjoint paths. Altogether we have chosen a set of $\left( \sum_{i=1}^{p} k_i \right) + q$ vertices such that deleting them from $G$ leaves $\sum_{i=1}^{p} \ell_i$  vertex-disjoint paths. That is, 

\begin{align*}
\Delta(G) & \geq \left( \displaystyle\sum_{i=1}^{p} \ell_i \right) - \left( \left( \displaystyle\sum_{i=1}^{p} k_i \right) + q \right)\\
		& = \left( \displaystyle\sum_{i=1}^{p} \left( \ell_i - k_i \right) \right) - q\\
		& = \left( \displaystyle\sum_{i=1}^{p} \Delta(T_i) \right) - q\\
		& = \left( \displaystyle\sum_{i=1}^{p} P(T_i) \right) - q\\
		& = T^-(G).
\end{align*}
The last equality above holds since the $q$ vertices were chosen to be an optimal set of vertices for $T^-(G)$.
\end{proof}

Note that since $\Delta(G)=T^-(G)$ for all graphs $G$, trees can be replaced by paths in the definition of $T^-(G)$. But this is not the case for $T^+(G)$. For example, for the graph $G$ in Figure \ref{eichgraph}, we have  $T^+(G) = P(G) =2$ and $\Delta^+(G) = 4$.

It is interesting to note that $T^+(G)$ is not only just upper bound for $P(G)$, but also for $M(G)$.

\begin{figure}[h]
\begin{center}
\begin{tikzpicture}
	\node[draw, circle, fill, black, scale = .5] (1) at (-1,0) {};
	\node[draw, circle, fill, black, scale = .5] (2) at (0,0) {};
	\node[draw, circle, fill, black, scale = .5] (3) at (1,0) {};
	\node[draw, circle, fill, black, scale = .5] (4) at (-1,-1) {};
	\node[draw, circle, fill, black, scale = .5] (5) at (0,-1) {};
	\node[draw, circle, fill, black, scale = .5] (6) at (1,-1) {};
	\path[draw] (1) -- (2) -- (3);
	\path[draw] (2) -- (5);
	\path[draw] (4) -- (5) -- (6);
\end{tikzpicture}
\end{center}
\caption{Graph $G$ with $T^+(G)=P(G)=2$ and $\Delta^+(G)=4$}\label{eichgraph}
\end{figure}
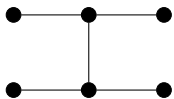

\begin{theorem}\label{Mislessthantplus}
For any graph $G$, \[M(G) \leq T^+(G).\]
\end{theorem}
\begin{proof}
The proof is similar to that of Theorem \ref{delta}.
\end{proof}

Now we discuss the connection of $T^+(G)$ with the zero forcing number $Z(G)$ which is an well-known upper bound for $M(G)$. A zero forcing set $Z$ is a subset of vertices of $G$ such that if the vertices in $Z$ are initially colored and the other vertices are not colored, then all the vertices of $G$ become colored after we apply the following coloring rule: if $u$ is a colored vertex with exactly one uncolored neighbor $v$, then color $v$.  The zero forcing number $Z(G)$ is the minimum size of a zero forcing set of $G$. For example, in the graph $G$ in Figure \ref{zfgraph}, two pendant vertices at distance 3 form a zero forcing set  because if they  are initially colored, then they will force the remaining vertices to be colored. Thus $Z(G)\leq 2$. Since there is no zero forcing set of size $1$, we have $Z(G)=2$.

Let $S$ be an optimal set for $T^+(G)$. Then $G\setminus S$ is a forest for which $P(G\setminus S)=M(G\setminus S)=Z(G\setminus S)$. Now we can find a zero forcing set $Z'$ of $G\setminus S$ of size $P(G\setminus S)$ by choosing an endpoint of each path in a minimum path cover of $G\setminus S$. It can be verified that $Z'\cup S$ is a zero forcing set of $G$ and consequently $$Z(G)\leq |Z'\cup S|=|Z'|+|S|=P(G\setminus S)+|S|=T^+(G).$$ 
Although $T^+(G)$ does not improve the upper bound $Z(G)$ of $M(G)$, it has a different approach than the zero forcing number $Z(G)$.

\begin{obs}\label{zf}
For any graph $G$, $Z(G) \leq T^+(G)$.
\end{obs} 
Note that the equality does not hold for the graph in Figure \ref{zfgraph}, where $Z(G)=P(G)=2$ and $T^+(G)=4$.
\begin{figure}[h]
\begin{center}
\begin{tikzpicture}
	\node[draw, circle, fill, black, scale = .5] (1) at (-1,0) {};
	\node[draw, circle, fill, black, scale = .5] (2) at (0,0) {};
	\node[draw, circle, fill, black, scale = .5] (3) at (1,0) {};
	\node[draw, circle, fill, black, scale = .5] (7) at (2,0) {};
	\node[draw, circle, fill, black, scale = .5] (4) at (-1,-1) {};
	\node[draw, circle, fill, black, scale = .5] (5) at (0,-1) {};
	\node[draw, circle, fill, black, scale = .5] (6) at (1,-1) {};
	\node[draw, circle, fill, black, scale = .5] (8) at (2,-1) {};
	\path[draw] (1) -- (2) -- (3)--(7);
	\path[draw] (2) -- (5);
	\path[draw] (3) -- (6);
	\path[draw] (4) -- (5) -- (6)--(8);
\end{tikzpicture}
\end{center}
\caption{Graph $G$ with $Z(G) < T^+(G)$.}\label{zfgraph}
\end{figure}
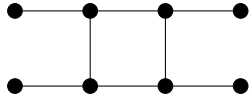

Now we summarize earlier results involving different parameters in the following corollary. 

\begin{corollary}\label{allinequalities}
For all graphs $G$,
\[\Delta(G) = T^-(G) \leq M(G) \leq Z(G) \leq T^+(G) \leq \Delta^+(G),\] where $T^-(G) = M(G) = T^+(G)$ if $G$ is a forest. 
\end{corollary}
\begin{proof}
By Theorem \ref{prop2}, we have $\Delta(G) = T^-(G)$. By Theorem \ref{knownbounds}, we have $\Delta(G) \leq M(G)$, hence $T^-(G) \leq M(G)$. By Theorem \ref{Mislessthantplus} and Observation \ref{zf}, $M(G) \leq Z(G)\leq  T^+(G)$. Finally, by Proposition \ref{prop1}(b), $T^+(G) \leq \Delta^+(G)$.

Note that if $G$ is a forest, then an optimal set of vertices to be deleted for $T^-$ and for $T^+$ can be chosen to be the empty set. Hence $q = 0$ and $\sum_{i=1}^{p} P(T_i) = P(G)$. Hence, $T^-(G) = T^+(G) = P(G) = M(G)$.
\end{proof}

Note that since $T^-(H_5)=M(H_5)$ and $M(W_5)=T^+(W_5)$, $T^-(G)$ and $T^+(G)$ give a tight lower bound and a tight upper bound for $M(G)$ respectively.

\section{On computing $T^-$ and $T^+$}\label{sectioncomputing}
In this section we provide some tools to reduce the time of computation for $T^-$ and $T^+$ for graphs. Let $G$ be a graph on $n$ vertices with $m$ edges. We show that there is always an optimal set of vertices to be deleted for $T^-$ and $T^+$ that has size at most $m-n+1$.

An {\it Eulerian subgraph} of $G$ is a subgraph of $G$ whose vertices have even degree. It is well-known that an Eulerian subgraph $H$ is a union of cycles in $H$. The (binary) {\it cycle space} of $G$ is the set of Eulerian subgraphs in  $G$. The cycle space of $G$ can be described as a vector space over $\mathbb{Z}_2$. A basis for this vector space is called a cycle basis of the graph. It can be shown that the dimension of the cycle space of a connected graph is $m-n+1$ \cite[Section 1.9 pp. 23--28]{diestel12}. Therefore any cycle in $G$ is a linear combination of cycles in a cycle basis and each of $m-n+1$ cycles in a cycle basis is not a linear combination of smaller cycles.

\begin{lemma}\label{cyclebasismanyverticesisenough}
Let $G$ be a connected graph on $n$ vertices with $m$ edges.  Let $S$ be a set of vertices of $G$ such that $G\setminus S$ does not have any cycles. Then there is a set $S' \subseteq S$ with $|S'| \leq m-n+1$ such that $G \setminus S'$ does not have any cycles.
\end{lemma}
\begin{proof}
If $|S| \leq m-n+1$, then choose $S' = S$. Otherwise, fix a cycle basis of $G$: \[B =\{ C_1, C_2, \ldots, C_{m-n+1} \}. \]
Let $S_1$ be the subset of $S$ consisting of vertices $v$ that is on exactly one cycle in $B$. Note that $S_1$ might be empty.
 Let $B_1$ be the subset of $B$ corresponding to vertices of $S_1$. Note that $|S_1|=|B_1|$ and a cycle basis of $G\setminus S_1$ is $B\setminus B_1$. Let $S_2$ be the subset of $S\setminus S_1$ consisting of at most one vertex from each cycle in $B\setminus B_1$ such that $(G\setminus S_1)\setminus S_2$ does not have any cycle. Choose $S'=S_1\cup S_2$.  Then $G\setminus S'$ does not have any cycles and $|S'|=|S_1|+|S_2|\leq |B_1|+|B\setminus B_1|=|B|=m-n+1$.
\end{proof}

\begin{lemma}\label{anysubsetofSthatleavesnocycleisenough}
Let $G$ be a graph on $n$ vertices with $m$ edges. Let $S$ be an optimal set of vertices for $T^+(G)$ (respectively $T^-(G)$). Then any set $S'\subseteq S$ such that $G\setminus S'$ does not have a cycle is also an optimal set of vertices for $T^+(G)$ (respectively $T^-(G)$).
\end{lemma}
\begin{proof}
First note, by the definitions of $T^+(G)$ (respectively $T^-(G)$), that $G\setminus S$ is a forest  and $T^+(G) = T^+(G\setminus S)+|S|$ (respectively $T^-(G) = T^-(G\setminus S)-|S|$). 

Consider $S' \subseteq S$ such that $G\setminus S'$ is a forest $F$. Let $R = S \setminus S'$. Choose an optimal path cover of  $F\setminus R =G\setminus S$: \[ P = \{P_1, P_2, \ldots, P_k\}.\] Then $P$ together with vertices in $R$ forms a path cover for $F$ which implies
\[ T^+(F\setminus R) + |R| \geq T^+(F). \]
If $T^+(F\setminus R) + |R| > T^+(F)$, then 
\[ T^+(G)=T^+(G\setminus S)+|S| = T^+(F\setminus R) + |R| +|S'| > T^+(F)+|S'|. \] This contradicts the optimality of $S$. Thus, \[ T^+(F\setminus R) + |R| = T^+(F). \] Consequently, \[ T^+(G) = T^+(F\setminus R) + |R| +|S'| = T^+(F)+|S'|. \] That is, $S'$ is also an optimal set of vertices for $T^+(G)$. 

Similarly, for $T^-(G)$ consider $S' \subseteq S$ such that $G\setminus S'$ is a forest $F$. Let $R = S \setminus S'$. Choose an optimal path cover of $F\setminus R=G\setminus S$: \[ P = \{P_1, P_2, \ldots, P_k\}.\] Then $P$ together with vertices in $R$ forms a path cover for $F$ which implies
\[ T^-(F\setminus R) - |R| \leq T^-(F). \]
If $T^-(F\setminus R) - |R| < T^-(F)$, then 
\[ T^-(G)=T^-(G\setminus S)-|S| = T^-(F\setminus R) - |R| -|S'| < T^-(F)-|S'|. \] This contradicts the optimality of $S$. Thus, \[ T^-(F\setminus R) - |R| = T^-(F). \] Consequently, \[ T^-(G) = T^-(F\setminus R) - |R| -|S'| = T^-(F)-|S'|. \] That is, $S'$ is also an optimal set of vertices for $T^-(G)$.
\end{proof}

\begin{prop}
Let $G$ be a connected graph on $n$ vertices with $m$ edges. The optimal sets of vertices for $T^+$ and $T^-$ can be chosen so that each of them has at most $m-n+1$ vertices.
\end{prop}
\begin{proof}
Consider an optimal set $S$ of vertices for $T^+(G)$ (respectively $T^-(G)$). Then $G\setminus S$ is a forest. By Lemma \ref{cyclebasismanyverticesisenough}, there is $S' \subseteq S$ such that $|S'| \leq m-n+1$ and $G\setminus S'$ does not have any cycles. Then by Lemma \ref{anysubsetofSthatleavesnocycleisenough}, $S'$ is also an optimal set of vertices for $T^+(G)$ (respectively $T^-(G)$).
\end{proof}

\section{Open Problems}\label{sectionproblems}
Recall from Corollary \ref{allinequalities} and Proposition \ref{prop1}(b) that $T^-(G)$ and $T^+(G)$ are lower and upper bounds for both $M(G)$ and $P(G)$ respectively. Therefore if $T^-(G)=T^+(G)$, then $$T^-(G)=\Delta(G)=P(G)=M(G)=Z(G)=T^+(G).$$
So it is natural to seek characterization of graphs $G$ for which $T^-(G)=T^+(G)$. 

\begin{question}\label{conjecture}
For what graphs $G$, $T^-(G) = T^+(G)$?
\end{question}

Note that $T^-(G)=T^+(G)$ for all forests (see Corollary \ref{allinequalities}) and  all unicyclic graphs in Example \ref{unicyclic}.\\



By Theorem \ref{Mislessthantplus}, $M(G) \leq T^+(G)$ for all graphs $G$ and the equality holds for forests (see Corollary \ref{allinequalities}), cycles (see Examples \ref{cycle}, \ref{tpluscycle}), wheel graphs (see Example \ref{wheel}), and  all unicyclic graphs in Example \ref{unicyclic}. It may be interesting to know what other graphs give the equality.

\begin{question}
Characterize graphs $G$ for which $M(G)=T^+(G)$. 
\end{question}

Similarly by Proposition \ref{prop1}(b), $P(G) \leq T^+(G)$ for all graphs $G$ and the equality holds for forests (see Corollary \ref{allinequalities}) and  all unicyclic graphs in Example \ref{unicyclic}. It may be worth exploring the graphs for which the equality holds.

\begin{question}
Characterize graphs $G$ for which $P(G)=T^+(G)$.
\end{question}

Finally, note that for any graph $G$ we have $Z(G) \leq T^+(G)$ (Observation \ref{zf}), while the equality holds for forests (see Corollary \ref{allinequalities}), cycles (see Examples \ref{cycle}, \ref{tpluscycle}), wheel graphs (see Example \ref{wheel}), and  all unicyclic graphs in Example \ref{unicyclic}. But the equality does not always hold (see Figure \ref{zfgraph}). It might be of interest to classify graphs for which we have the equality.

\begin{question}
	Characterize graphs $G$ for which $Z(G)=T^+(G)$.
\end{question}

\section*{Acknowledgment}
An Acknowledgment will be added later.

\bibliographystyle{plain}
\bibliography{ref}

\end{document}